\documentclass[12pt]{article}
\usepackage[dvips]{epsfig}
\usepackage{amsmath}
\usepackage{amsfonts}
\usepackage{amssymb}
\usepackage{graphicx}
\usepackage{amsmath,amsfonts,amssymb,mathrsfs,amscd}
\usepackage{psfrag}
\usepackage{bbm}
\usepackage{float}
\usepackage{latexsym}
\usepackage{CJK}
\newtheorem{thm}{Theorem}[section]
\newtheorem{coro}{Corollary}[section]
\newtheorem{lm}{Lemma}[section]
\newtheorem{prop}{Proposition}[section]

\newtheorem{defn}{Definition}[section]

\numberwithin{equation}{section}
\newenvironment{proof}[1][Proof]{\noindent\textbf{#1.} }{\ \rule{0.5em}{0.5em}}

\renewcommand{\phi}{\varphi}

\renewcommand{\chi}{\mathcal{X}}

\newcommand{\tr}{\rm{tr}}

\begin{document}

\title{On Variational Properties of Quadratic Curvature Functionals}
\author{{ SHENG Weimin  \ \ and \ \ WANG Lisheng } }
\date{}
\maketitle

\begin{abstract}
  In this paper, we investigate a class of quadratic Riemannian curvature functionals on closed smooth manifold $M$ of dimension $n\ge 3$ on the space of Riemannian metrics on $M$ with unit volume. We study the stability of these functionals at the metric with constant sectional curvature as its critical point. 

\end{abstract}

{\footnotetext{\baselineskip 10pt
%$^\dag$ Corresponding author\\
The authors were supported by NSFC 11571304.}

\vspace{1mm}{\footnotesize \textbf{Keywords}:\hspace{2mm} Quadratic curvature functional, Variational, Transverse-traceless, Conformal variation, Stability
 }

{\footnotesize{\bf MSC(2000):\hspace{2mm} 58E11, 53C24}
\vspace{2mm} \baselineskip 15pt \renewcommand{\baselinestretch}{1.22}
\parindent=10.8pt

\section{Introduction}

 Let $M$ be an $n$-dimensional compact and smooth manifold, and $\mathscr{M}_1$ the space of smooth Riemannian metrics on $M$  with unit volume, i.e. $\mathscr{M}_1=\left\{g\in \mathscr{M}: vol(g)=1\right\}$, where $\mathscr{M}$ is the space of smooth Riemannian metrics on $M$.  A functional $F: \mathscr{M}\rightarrow R$ is called Riemannian if it is invariant under  the action of the diffeomorphism group. 

 Recall the decomposition of Riemannian curvature tensor $Rm$
 \begin{equation}
 Rm=W+\frac{1}{n-2}Ric\odot g-\frac{1}{(n-1)(n-2)}Rg\odot g, \label{decomposition}
\end{equation}
where $W$, $Ric$ and $R$ denote the Weyl curvature tensor, the Ricci tensor and the scalar curvature, respectively, and $\odot$ the Kulkarni-Nomizu product. From \eqref{decomposition},
we have,
\begin{equation}
|Rm|^2=|W|^2+\frac{4}{n-2}|Ric|^2 -\frac{2}{(n-1)(n-2)}R^2. \label{decomposition1}
\end{equation}
 The basic quadratic curvature functionals are
 \[
 \mathcal{W}=\int_M|W|^2dV_g, \ \ \ \rho=\int_M|Ric|^2dV_g, \ \ \ \mathcal{S}=\int_M R^2dV_g.
 \]
From the decomposition formula \eqref{decomposition1}, one has 
\[
\mathcal{R}=\int_M|Rm|^2dV_g=\int_M|W|^2dV_g+\frac{4}{n-2}\int_M|Ric|^2dV_g -\frac{2}{(n-1)(n-2)}\int_MR^2dV_g.
\]

We point out that in dimension three, Weyl tensor vanishes, and in dimension four, the famous Chern-Gauss-Bonnet formula implies that $\mathcal{W}$ can be expressed as a linear combination of $\rho$ and $\mathcal{S}$ with the addition of a topological term. There are many results on these quadratic functionals. See \cite{Besse, Blair2000, Berger1969, Viaclovsky2014} for example. In \cite{GV2003-2}, Gursky and Viaclovsky focus attention on a class of general quadratic curvature functionals
\[
\mathcal{\widetilde{F}}_\tau=(Vol)^{\frac{4-n}{n}}\big(\int_M|Ric|^2dV_g+\tau \int_M|R|^2dV_g\big).
\]
They investigate rigidity and stability properties of critical points of $\mathcal{\widetilde{F}}_\tau$
on the space of Riemannian metric space $\mathscr{M}$, and obtain a series of beautiful results.

In this paper, we study a class of more general quadratic curvature functionals:
\begin{equation}
\mathcal{F}_{s,\tau}=\int_M|Rm|^2dV_g+s\int_M|Ric|^2dV_g+\tau\int_MR^2dV_g\ \label{F-stau}
\end{equation}
on Riemannian metrics space $\mathscr{M}_1$, where $s,\tau$ are some constants.

Actually, $\mathcal{F}_{s,\tau}$ have been widely studied. We first introduce the following definition. 

\begin{defn}
Let $M$ be a compact $n$-dimension manifold, a critical metric $g$ for $\mathcal{F}_{s,\tau}$ is called a local minimizer if for all metrics $\bar{g}$ in a $C^{2,\alpha}$-neighborhood of $g$, satisfying:
\[
\mathcal{F}_{s,\tau}[\bar{g}]\geq \mathcal{F}_{s,\tau}[g].
\]
\end{defn}
\quad We can also define the local maximizer for $\mathcal{F}_{s,\tau}$ by the same way.

In \cite{Muto1974}, Y. Muto studied Riemannian functional $I[g]=\int_M|Rm|^2dV_g$ on $\mathscr{M}_1$,  and proved when $M$ is a $C^\infty$ manifold diffeomorphic to $S^n$, the mapping ${I}:\mathscr{M}_1/\mathcal{D}\rightarrow R$  has a local minimum at the Riemannian metric $\bar{g}$ of positive constant sectional curvature, 
where ${I}:\mathscr{M}_1/\mathcal{D}\rightarrow R$ is a mapping deduced from $I:\mathscr{M}_1\rightarrow R$  and $\mathcal{D}$ is the diffeomorphism group of $M$ and $\mathscr{M}_1/\mathcal{D}$ is the space of orbits generated by $\mathcal{D}$ of Riemannian metrics. Recently, S. Maity \cite{Maity2013} generalized the result for functional $\mathcal{R}_p(g)=\int_M|Rm|^pdV_g$ on  $\mathscr{M}_1$, she proved that for a compact Riemannian manifold $(M,g)$, if $g$ is either a spherical space form and $p\in [2,\infty)$, or a hyperbolic manifold and $p\in [\frac{n}{2},\infty)$, then $g$ is strictly stable for  $\mathcal{R}_p$. O. Kobayashi \cite{Kobayashi1984} investigated variational properties of the conformally invariant functional $v(g)=\frac{2}{n}\int_M|W|^\frac{n}{2}$, and derived that when $n=4$, $S^2(1)\times S^2(1)$ endowed with the standard Einstein metric $g$ is a strictly stable critical point of $v(g)$. In \cite{Guo2015}, X. Guo, H. Li and G. Wei developed the result to the case of dimension $n=6$.

In this paper,  we concern the various properties for quadratic curvature functional  $\mathcal{F}_{s,\tau}$. Before giving our results, we state some background knowledge.

 Recall a canonical decomposition of tangent space of $\mathscr{M}$. Define the divergence operator $\delta_g: S^2(M)\rightarrow T^*M$ by
 \[
 (\delta_g h)_j=g^{pq}h_{pj,q}.
 \]
 and $\delta^*_g:  T^*M\rightarrow S^2(M)$ its $L^2-$adjoint operator
 \[
 (\delta^*_g\omega)_{ij}=-\frac{1}{2}(\omega_{i,j}+\omega_{j,i}),
 \]
 in local coordinates $\{e_i\}_{1\le i\le n}$, where $h_{ij,k}=\nabla_kh_{ij}$ and $\omega_{i,j}=\nabla_j\omega_i$, $\nabla$ is the Riemannian covariant derivative of $(M, g)$.
 For a compact Riemannian manifold $M$, the tangent space of $\mathscr{M}$ at $g$, which is denoted by $T_g\mathscr{M}$ has the orthogonal decomposition (see \cite{Besse} Lemma 4.57):
 \begin{equation}
 T_g\mathscr{M}=S^2(M)=(Im\delta_g^*+C^\infty(M)\cdot g)\oplus(\delta_g^{-1}(0)\cap \tr_g^{-1}(0)),\label{tangentdecomp}
 \end{equation}
where $Im\delta_g^*$ is just the tangent space of the orbit of $g$ under the action of the group of diffeomorphisms of $M$. For $T_g\mathscr{M}_1=\{h\in S^2(M): \int_Mtr_g(h)dV_g=0\}$, we have:
 \begin{equation}
 T_g\mathscr{M}_1=\left((Im\delta_g^*+C^\infty(M)\cdot g)\cap T_g\mathscr{M}_1\right)\oplus(\delta_g^{-1}(0)\cap \tr_g^{-1}(0)).\label{tangentdecomp1}
 \end{equation}

\begin{defn}
Let $h\in S^2(M)$, then $h$ is called transverse-traceless (TT for short) if $\delta_gh=0$ and $\tr_gh=0$.
\end{defn}

On an Einstein manifold, by \cite{Lich1961}, we have the Lichnerowicz Laplacian
\begin{equation}
\triangle_Lh_{ij}=\triangle h_{ij}+2R_{ikjl}h^{kl}-\frac{2}{n}Rh_{ij},\label{Lichnerowicz}
 \end{equation}
where $\triangle$ is the rough Laplacian, and $\triangle_L$ maps the space of transverse-traceless tensors to iteself (see \cite{GV2003-2} also). By use of \ref{Lichnerowicz}, we can get following Propositions 1.1-1.2 easily. 

\begin{prop} Let $(M,g)$ be a compact manifold with constant sectional curvature $1$. Then the least eigenvalue of the Lichnerowicz Laplacian on TT-tensors is $4n$.\label{LLposieigen}
\end{prop}
\begin{proof} By \eqref{Lichnerowicz}, the Lichnerowicz Laplacian on TT tensor $h$ is 
\[
\triangle_Lh=\triangle h-2nh.
\]
Then by use of the  inequality
\[
0\le \int_M|h_{ij,k}+h_{jk,i}+h_{ki,j}|^2dV_g,
\]
exchanging covariant derivatives and integrating by parts, we have that the least eigenvalue of the Lichnerowicz Laplacian on TT-tensors is $4n$.
\end{proof}

\begin{prop}Let $(M,g)$ be a compact manifold with constant sectional curvature $-1$. Then the least eigenvalue of the Lichnerowicz Laplacian on TT-tensors is bounded below by $-n$.
\label{negeigen}
\end{prop}
\begin{proof} Make use of  the inequality 
\[
\int_M|h_{ij,k}-h_{ik,j}|^2dV_g\geq 0, 
\]
 we can get the result in the same way. 
\end{proof}

It can be easily to see (Corollary 2.2 below) that the metric $g$ with constant sectional curvature must be a critical point of $\mathcal{F}_{s,\tau}$. In this paper we study the stability of $\mathcal{F}_{s,\tau}$ at the critical point $g$ with constant sectional curvature. 
 We have the following 
\begin{thm}Let $(M,g)$ be a $n$-dimensional compact manifold with constant sectional curvature $\lambda=1$, then restricted to transverse-traceless variations, the following hold:
\begin{itemize}
 \item[(1)] If $s>-4$, $\tau<\frac{6n-12}{n(n-1)}$, the second variation of $\mathcal{F}_{s,\tau}$ on $(M,g)$ is non-negative. Therefore $\mathcal{F}_{s,\tau}$ gets its local minimum in TT directions; 
\item[(2)] If $s<-4$, $\tau>\frac{6n-12}{n(n-1)}$, the second  variation of $\mathcal{F}_{s,\tau}$ on $(M,g)$ is non-positive. Therefore $\mathcal{F}_{s,\tau}$ gets its local maximum in TT directions.
\end{itemize}
\end{thm}
\begin{thm} Let $(M,g)$ be a $n$-dimensional compact manifold with constant sectional curvature $\lambda=-1$, then restricted to transverse-traceless variations, the following hold:
\begin{itemize}
\item[(1)] If $s>-4$, $\tau>\frac{6n-12}{n(n-1)}$, the second variation of $\mathcal{F}_{s,\tau}$ on $(M,g)$ is non-negative. Therefore $\mathcal{F}_{s,\tau}$ gets its local minimum in TT directions; 
\item[(2)] If $s<-4$, $\tau<\frac{6n-12}{n(n-1)}$, the second variation of $\mathcal{F}_{s,\tau}$ on $(M,g)$ is non-positive. Therefore $\mathcal{F}_{s,\tau}$ gets its local maximum in TT directions.
\end{itemize}
\end{thm}

\begin{thm} Let $(M,g)$ be a compact manifold  with constant sectional curvature $\lambda=0$. Then the second variation of $\mathcal{F}_{s,\tau}$ at $g$ is non-negative as $s>-4$ and non-positive as $s<-4$ when the variation is restricted in TT directions.
\end{thm}

\begin{thm} Let $(M,g)$ be a compact manifold  with constant sectional curvature $\lambda=0$. Then the second variation of $\mathcal{F}_{s,\tau}$ at $(M,g)$ is nonnegative as $s+4\tau>\frac{4}{n}(\tau-1)$ and non-positive as $s+4\tau<\frac{4}{n}(\tau-1)$ when the variation is restricted in the conformal directions.
\end{thm}

We denote $\mathcal{M}_1([g])$  the space of unit volume metrics conformal to $g$.

\begin{thm} Let $(M,g)$ be an $n$-dimensional compact manifold with constant sectional curvature $\lambda=1$.  Then the following hold:
\begin{itemize}
\item[(1)] If $n=4$, $s+3\tau>-1$, then $\mathcal{F}_{s,\tau}$ attains a local minimizer at $g$  in $\mathcal{M}_1([g])$. If $s+3\tau<-1$, then $g$ is a local maximizer in $\mathcal{M}_1([g])$.
\item[(2)] If $n=3$, $\tau<1$ and  $s>-\frac{8}{3}\tau-\frac{4}{3}$, or $\tau>1$ and
$s>-\frac{12}{5}\tau-\frac{8}{5}$, then $\mathcal{F}_{s,\tau}$ attains a local minimizer at $g$  in $\mathcal{M}_1([g])$. If if $\tau<1$ and $s<-\frac{12}{5}\tau-\frac{8}{5}$, or $\tau>1$ and
$s<-\frac{8}{3}\tau-\frac{4}{3}$, then $g$ is a local maximizer in $\mathcal{M}_1([g])$.
\item[(3)] When $n\geq 5$, then $\mathcal{F}_{s,\tau}$ attains a local minimizer at $g$  in $\mathcal{M}_1([g])$ if
\begin{equation}
\left\{
\begin{array}{lr}
\tau>\frac{2}{(n-1)(n-2)} &\\
s>-\frac{4(n-1)}{n}\tau-\frac{4}{n}&
\end{array}
\right.
\quad or\quad
\left\{
\begin{array}{lr}
\tau<\frac{2}{(n-1)(n-2)}&\\
s>-\frac{2n(n-1)}{3n-4}\tau-\frac{8}{3n-4}.&
\end{array}
\right.
\end{equation}
and $g$ is a local maximizer in $\mathcal{M}_1([g])$ if
\begin{equation}
\left\{
\begin{array}{lr}
\tau>\frac{2}{(n-1)(n-2)}&\\
s<-\frac{2n(n-1)}{3n-4}\tau-\frac{8}{3n-4}
\end{array}
\right.
\quad or\quad
\left\{
\begin{array}{lr}
\tau<\frac{2}{(n-1)(n-2)}&\\
s<-\frac{4(n-1)}{n}\tau-\frac{4}{n}.&
\end{array}
\right.
\end{equation}
\end{itemize}
\end{thm}

\begin{thm} Let $(M,g)$ be an $n$-dimensional compact manifold with constant sectional curvature $\lambda=-1$,  the following hold:

\quad (1) If $n=4$, $s+3\tau>-1$, Then $\mathcal{F}_{s,\tau}$ attains a local minimizer at $g$ in $\mathcal{M}_1([g])$. If $s+3\tau<-1$, then $g$ is a local maximizer in $\mathcal{M}_1([g])$.

\quad (2) If $n=3$, $\tau<1$ and  $s>-\frac{8}{3}\tau-\frac{4}{3}$, or $\tau>1$ and
$s>-\frac{12}{5}\tau-\frac{8}{5}$, then $\mathcal{F}_{s,\tau}$ attains a local minimizer at $g$  in $\mathcal{M}_1([g])$. If  $\tau<1$ and $s<-\frac{12}{5}\tau-\frac{8}{5}$, or $\tau>1$ and
$s<-\frac{8}{3}\tau-\frac{4}{3}$, then $g$ is a local maximizer in $\mathcal{M}_1([g])$.

\quad (3) When $n\geq 5$, then $\mathcal{F}_{s,\tau}$ attains a local minimizer at $g$ in $\mathcal{M}_1([g])$ if
\begin{equation}
\left\{
\begin{array}{lr}
\tau<\frac{2}{(n-1)(n-2)} &\\
-\frac{4(n-1)}{n}\tau-\frac{4}{n}<s<-n\tau-\frac{2}{n-1}.&
\end{array}
\right.
\end{equation}
and $g$ is a local maximizer in $\mathcal{M}_1([g])$ if
\begin{equation}
\left\{
\begin{array}{lr}
\tau<\frac{2}{(n-1)(n-2)}&\\
s>-n\tau-\frac{2}{n-1}
\end{array}
\right.
\quad or\quad
\left\{
\begin{array}{lr}
\tau>\frac{2}{(n-1)(n-2)}&\\
s>-\frac{4(n-1)}{n}\tau-\frac{4}{n}.&
\end{array}
\right.
\end{equation}
\end{thm}

%%%%%%%%%%%%%%%%%%%%%%%%%%%%%%%%%%%

\section{The first variation and Euler-Lagrange equation}

 Suppose $(M,g)$ be a $n$-dimensional Riemannian manifold.  We choose a local orthonormal frame $\left\{e_1, e_2, \ldots e_n\right\}$, in accordance with the dual coframe $\left\{\omega^1, \omega^2, \ldots \omega^n\right\}$. Throughout this paper, we always adopt the moving frame notation with respect to a chosen local orthonormal frame, and also the Einstein summation convention. Let $g\in \mathscr{M}$ be an arbitrary fixed metric and $g=g_{ij}\omega^i\otimes\omega^j$, $g^{ij}=(g_{ij})^{-1}$. Given a tensor, we raise  or lower an index by contracting the tensor with the metric tensor $g$.

We denote $\nabla$ as the covariant derivative, and write $R_{ij,k}=\nabla_kR_{ij}$, $R_{ij,kl}=\nabla_l\nabla_kR_{ij}$, the Laplacian $\triangle R_{ij}=g^{kl}R_{ij,kl}$ and so on. For any $(0,2)$ tensor $S$, the Ricci identities can be expressed as 
\[
S_{ij,kl}-S_{ij,lk}=S_{pj}R_{pikl}+S_{ip}R_{pjkl}.
\]

Let $g(t)\in\mathscr{M}_1$ be a smooth variation of $g$ with $g(0)=g$ which can be represented locally as $g_{ij}(x^1,\ldots,x^n;t)$.  We define a tensor field $h\in S^2(M)$ with  $h := \frac{d}{dt}g(t)$. For convenience, we write $(\cdot)'$ to stand for
$\frac{d}{dt}$. Then we have the following formulae
 \begin{equation}
 (g_{ij})'=h_{ij}, \ \ \   \ (g^{ij})'=-h^{ij}. \label{derivative}
 \end{equation}
 
\begin{prop} Let $g(t)\in\mathscr{M}_1$ is a smooth variation \eqref{derivative}, then 
\begin{equation}\label{2.2}
\int_M \tr_{g(t)}hdV_g=0,\quad
\int_M \left(g^{ij}\frac{d^2}{dt^2}g_{ij}-|h|^2+\frac{1}{2}H^2\right)dV_g=0,
\end{equation}
where $|h|^2=h^{ij}h_{ij}$, $H=\tr_{g(t)}h=g^{ij}h_{ij}$.
\end{prop}

\begin{proof}
Recall $g(t)\in\mathscr{M}_1$, we have $\int_M dV_g=1$, then we have 
\[
0=\frac{d}{dt}(\int_M dV_g)=\int_M \frac{d}{dt}(dV_g)=\int_M \frac{1}{2}g^{ij}(g_{ij})'dV_g=\frac{1}{2}\int_M \tr_{g(t)}hdV_g.
\]
Differentiating the above equality, we get
\[
\int_M \left(g^{ij}\frac{d^2}{dt^2}g_{ij}-|h|^2+\frac{1}{2}(\tr_{g(t)}h)^2\right)dV_g=0.
\]
This proves \eqref{2.2}.
\end{proof}

From \eqref{derivative}, we get the variation of the Christoffel symbols£¬
\begin{equation}
\frac{d}{dt}\Gamma^{^k}_{ij}=\frac{1}{2}g^{kl}(h_{il,j}+h_{jl,i}-h_{ij,l}).\label{Christoffel}
\end{equation}
By use of \eqref{Christoffel}, we can get the following variational formulae of Riemannaian curvature tensor, Ricci curvature tensor and scalar curvature directly. 
\begin{prop}The variations of Riemannaian curvature tensor, Ricci curvature tensor and scalar curvature are expressed as 
 \begin{flalign*}
  (R^{l}_{\ ijk})'&=\frac{1}{2}g^{pl}\Big(h_{ip,kj}+h_{kp,ij}-h_{ik,pj}-h_{ip,jk}-h_{jp,ik}+h_{ij,pk}\Big), \\
   (R_{lijk})'&=h_{lq}R^{q}_{\ ijk}+\frac{1}{2}\Big(h_{il,kj}+h_{kl,ij}-h_{ik,lj}-h_{il,jk}-h_{jl,ik}+h_{ij,lk}\Big), \\
 (R_{ik})'&=\frac{1}{2}\Big(h^j_{i,kj}+h^j_{k,ij}-\triangle h_{ik}-H_{,ik}\Big),   \\
    R'&=-h^{ij}R_{ij}+h^{ij}_{\ ,ij}-\triangle H.  \\
 \end{flalign*}\label{prop2.2}
\end{prop}

Now, we compute the first variation of the quadratic curvature functional $\mathcal{F}_{s,\tau}$  restricting on Riemannian metrics space $\mathscr{M}_1$ and derive its Euler-Lagrange equation. At first, we compute the first variations of $\mathcal{R}$, $\rho$, and $\mathcal{S}$,  respectively.

By Proposition $\ref{prop2.2}$, we have
\begin{align*}
\frac{d}{dt}|Rm|^2&=\frac{d}{dt}\Big(g^{pl}g^{qi}g^{rj}g^{sk}R_{pqrs}R_{lijk}\Big)
\\&=-R_{p}^{\ ijk}R_{lijk}h^{pl}-R^{l\  jk}_{q}R_{lijk}h^{qi}-R_{\ r}^{li\ k}R_{lijk}h^{rj}-R_{\quad s}^{lij}R_{lijk}h^{sk}+2R^{lijk}\frac{d}{dt}R_{lijk}
\\&=-R_{p}^{\ ijk}R_{lijk}h^{pl}-R^{l\  jk}_{q}R_{lijk}h^{qi}-R_{\ r}^{li\ k}R_{lijk}h^{rj}-R_{\quad s}^{lij}R_{lijk}h^{sk}+2R^{lijk}h_{lq}R^{q}_{\ ijk}\\&\ \ \ \ +R^{lijk}\Big(h_{il,kj}+h_{kl,ij}-h_{ik,lj}-h_{il,jk}-h_{jl,ik}+h_{ij,lk}\Big)
\\&=-R_{p}^{\ ijk}R_{lijk}h^{pl}-R^{l\  jk}_{q}R_{lijk}h^{qi}-R_{\ r}^{li\ k}R_{lijk}h^{rj}-R_{\quad s}^{lij}R_{lijk}h^{sk}\\&\ \ \ \ +2R^{lijk}h_{lq}R^{q}_{\ ijk}+4R^{lijk}h_{ij,lk},  \\
\frac{d}{dt}|Ric|^2&=\frac{d}{dt}\Big(g^{pi}g^{qk}R_{pq}R_{ik}\Big)
\\&=-2h^{pi}R_{p}^{\ k}R_{ik}+2R^{ik}\frac{d}{dt}R_{ik}
\\&=-2h^{pi}R_{p}^{\ k}R_{ik}+R^{ik}\Big(h^j_{i,kj}+h^j_{k,ij}-\triangle h_{ik}-H_{,ik}\Big), \\
\frac{d}{dt}R^2&=2R\Big(-h^{ij}R_{ij}+h^{ij}_{\ ,ij}-\triangle H\Big).
\end{align*}
Integrating by parts, we get
\begin{align*}
\frac{d}{dt}\mathcal{R}&=\int_M \Big(\frac{d}{dt}|Rm|^2+\frac{1}{2}|Rm|^2H\Big)dV_g
\\&=\int_M \Big(-R_{p}^{\ ijk}R_{lijk}h^{pl}-R^{l\  jk}_{q}R_{lijk}h^{qi}-R_{\ r}^{li\ k}R_{lijk}h^{rj}-R_{\quad s}^{lij}R_{lijk}h^{sk}\\&\ \ \ \ +2R^{lijk}h_{lq}R^{q}_{\ ijk}+4R^{lijk}h_{ij,lk}+\frac{1}{2}|Rm|^2H\Big)dV_g
\\&=\int_M \Big(-2R_{i}^{\ plk}R_{jplk}+2R_{,ij}-4\triangle R_{ij}-4R^{pl}R_{ipjl}+4R_{jp}R_{i}^{p}+\frac{1}{2}|Rm|^2g_{ij}\Big)h^{ij}dV_g, \\
\frac{d}{dt}
\rho&=\int_M \Big(\frac{d}{dt}|Ric|^2+\frac{1}{2}|Ric|^2H\Big)dV_g
\\&=\int_M \Big(-2h^{pi}R_{p}^{\ k}R_{ik}+R^{ik}\big(h^j_{i,kj}+h^j_{k,ij}-\triangle h_{ik}-H_{,ik}\big)+\frac{1}{2}|Ric|^2H\Big)dV_g
\\&=\int_M \Big(-\triangle R_{ij}-2R^{pl}R_{ipjl}+R_{,ij}-\frac{1}{2}(\triangle R)g_{ij}+\frac{1}{2}|Ric|^2g_{ij}\Big)h^{ij}dV_g, \\
\frac{d}{dt}\mathcal{S}&=\int_M \Big(\frac{d}{dt}|R|^2+\frac{1}{2}|R|^2H\Big)dV_g
\\&=\int_M \Big(2R\big(-h^{ij}R_{ij}+h^{ij}_{\ ,ij}-\triangle H\big)+\frac{1}{2}|R|^2H\Big)dV_g
\\&=\int_M \Big(2R_{,ij}-2(\triangle R)g_{ij}-2RR_{ij}+\frac{1}{2}R^2g_{ij}\Big)h^{ij}dV_g.
\end{align*}

\begin{lm}(\cite{Besse}) The gradients of the functionals $\mathcal{R}$, $\rho$, $\mathcal{S}$, and $\mathcal{F}_{s,\tau}$ are given by the following equations.
\begin{align*}
(\nabla\mathcal{R})_{ij}&=-2R_{i}^{\ plk}R_{jplk}+2R_{,ij}-4\triangle R_{ij}-4R^{pl}R_{ipjl}+4R_{jp}R_{i}^{p}+\frac{1}{2}|Rm|^2g_{ij},\\
(\nabla\rho)_{ij}&=-\triangle R_{ij}-2R^{pl}R_{ipjl}+R_{,ij}-\frac{1}{2}(\triangle R)g_{ij}+\frac{1}{2}|Ric|^2g_{ij},\\
 (\nabla\mathcal{S})_{ij}&= 2R_{,ij}-2(\triangle R)g_{ij}-2RR_{ij}+\frac{1}{2}R^2g_{ij},\\
 (\nabla\mathcal{F}_{s,\tau})_{ij}&=-2R_{i}^{\ plk}R_{jplk}+2R_{,ij}-4\triangle R_{ij}-4R^{pl}R_{ipjl}+4R_{jp}R_{i}^{p}+\frac{1}{2}|Rm|^2g_{ij}\\&\ \ \ \ +s\Big(-\triangle R_{ij}-2R^{pl}R_{ipjl}+R_{,ij}-\frac{1}{2}(\triangle R)g_{ij}+\frac{1}{2}|Ric|^2g_{ij}\Big)\\&\ \ \ \ +\tau\Big(2R_{,ij}-2(\triangle R)g_{ij}-2RR_{ij}+\frac{1}{2}R^2g_{ij}\Big).\tag{2.5}\label{F-gradient}
 \end{align*}
\end{lm}

  By the Lagrangian multiplier method, a Riemannian metric $g\in \mathscr{M}_1$ is critical for $\mathcal{F}_{s,\tau}|\mathscr{M}_1$ if and only if it satisfies the equation 
\[
(\nabla\mathcal{F}_{s,\tau})_{ij}=cg_{ij}  
\tag{2.6}\label{Eul}
\]
for some constant $c$.
 The Euler-Lagrange equations for $\mathcal{F}_{s,\tau}|\mathscr{M}_1$ is obtained after a simple computation.
 
\begin{thm} Let $M$ be a compact $n$-dimensional Riemannian manifold. Then the Euler-Lagrange equations of $\mathcal{F}_{s,\tau}|\mathscr{M}_1$ are 
\begin{align*}
-(4+s)\triangle R_{ij}+(2+s+2\tau)R_{,ij}+\frac{2-2\tau}{n}\triangle Rg_{ij}-2R_{i}^{\ plk}R_{jplk}-(4+2s)R^{pl}R_{ipjl}\\
+4R_{jp}R_{i}^{p}-2\tau RR_{ij}+\frac{2}{n}\big(|Rm|^2g_{ij}+s|Ric|^2g_{ij}+\tau R^2g_{ij}\big)=0, \tag{2.7}\label{Eul1}\\
(n-4)\big(|Rm|^2+s|Ric|^2+\tau|R|^2\big)-(4+ns+4(n-1)\tau)\triangle R=2nc.\tag{2.8}\label{Eul2}
\end{align*}\label{Eulequ}
\end{thm}

In fact, \eqref{Eul2} comes from \eqref{Eul} by taking trace on both sides. Substituting \eqref{Eul2} into \eqref{Eul} we can get \eqref{Eul1} directly. 

From Theorem \ref{Eulequ}, we know that any compact Riemannian manifold $(M,g)$ with constant sectional curvature is a critical metric for $\mathcal{F}_{s,\tau}$ on $\mathscr{M}_1$. However, Einstein metric can not always be a critical point of these functionals.

\begin{coro}(\cite{Besse, Catino2015}) Restricting on $\mathscr{M}_1$, an Einstein metric $g$ is a critical point of $\mathcal{F}_{s,\tau}$ if and only if the metric $g$ satisfies
\[
R_{i}^{\ plk}R_{jplk}=\frac{1}{n}|Rm|^2g_{ij}.
\]
\end{coro}

\begin{coro} Any metric with constant sectional curvature is a critical point of $\mathcal{F}_{s,\tau}$ restricted on $\mathscr{M}_1(M^n)$.
\end{coro}

%%%%%%%%%%%%%%%%%%%%%%%%%%%%
\section{ Second variations on constant sectional curvature manifolds }

In this section, we derive the second variations of $\mathcal{F}_{s,\tau}|\mathscr{M}_1$ at the metric with constant sectional curvature. Suppose $(M,g)$ has constant sectional curvature, then for some constant $\lambda$,

\begin{equation}
R_{ijkl}=\lambda(g_{ik}g_{jl}-g_{il}g_{jk}).\label{Riemcurv}
\end{equation}
From \eqref{Riemcurv}, we also have
\begin{align}
R_{ij}=(n-1)\lambda g_{ij},\quad R=n(n-1)\lambda. \label{Riccurv}
\end{align}
Following \eqref{F-gradient}, the first variation of $\mathcal{F}_{s,\tau}$ is
\begin{align*}
 \frac{d}{dt}\mathcal{F}_{s,\tau}&=\int_M\Bigg(-2R_{i}^{\ plk}R_{jplk}+2R_{,ij}-4\triangle R_{ij}-4R^{pl}R_{ipjl}+4R_{jp}R_{i}^{p}+\frac{1}{2}|Rm|^2g_{ij}\\&\ \ \ \ +s\Big(-\triangle R_{ij}-2R^{pl}R_{ipjl}+R_{,ij}-\frac{1}{2}(\triangle R)g_{ij}+\frac{1}{2}|Ric|^2g_{ij}\Big)\\&\ \ \ \ +\tau\Big(2R_{,ij}-2(\triangle R)g_{ij}-2RR_{ij}+\frac{1}{2}R^2g_{ij}\Big)\Bigg)h^{ij}dV_g.\tag{3.3}\label{firstvariation}
\end{align*}
For convenience, we write $G=\nabla\mathcal{F}_{s,\tau}$,
\begin{align*}
G_{ij}&=-2R_{i}^{\ plk}R_{jplk}+2R_{,ij}-4\triangle R_{ij}-4R^{pl}R_{ipjl}+4R_{jp}R_{i}^{p}+\frac{1}{2}|Rm|^2g_{ij}\\&\ \ \ \ +s\Big(-\triangle R_{ij}-2R^{pl}R_{ipjl}+R_{,ij}-\frac{1}{2}(\triangle R)g_{ij}+\frac{1}{2}|Ric|^2g_{ij}\Big)\\&\ \ \ \ +\tau\Big(2R_{,ij}-2(\triangle R)g_{ij}-2RR_{ij}+\frac{1}{2}R^2g_{ij}\Big).\tag{3.4}\label{Gij}
\end{align*}
Then, we rewrite the first variation of $\mathcal{F}_{s,\tau}$ \eqref{firstvariation} as
\[
\frac{d}{dt}\mathcal{F}_{s,\tau}=\int_MG_{ij}h^{ij}dV_g. \tag{3.5}\label{LF}
\]

Differentiating  equality \eqref{LF} again, by use of  Proposition \ref{derivative}, we have
\begin{align*}
\frac{d^2}{dt^2}\bigg|_{t=0}&\mathcal{F}_{s,\tau}=\frac{d}{dt}\bigg|_{t=0}\int_MG_{ij}h^{ij}dV_g
\\&=\int_M\frac{d}{dt}\bigg|_{t=0}(G_{ij})h^{ij}dV_g+\int_M\Big(-2c|h|^2+cg^{ij}\frac{d^2}{dt^2}\bigg|_{t=0}(g_{ij})\Big)dV_g+\int_M\frac{1}{2}cH^{2}dV_g
\\&=\int_M\frac{d}{dt}\bigg|_{t=0}(G_{ij})h^{ij}dV_g-\int_Mc|h|^2dV_g.\tag{3.6}\label{secondvariation}
\end{align*}

Nextly, we concern the second variations on TT directions at some critical metric with constant sectional curvature.

\subsection{Transverse-traceless variations}

According to \eqref{Gij} and \eqref{secondvariation}, we computer $\frac{d}{dt}\big|_{t=0}(G_{ij})$ at  the metric $g$ with constant sectional curvature.
\begin{prop} If $g$ has constant sectional curvature, satisfying \eqref{Riemcurv} and \eqref{Riccurv}, then
\begin{align*}
(R_{ij,k})'&=(R_{ij}')_{,k}-\lambda(n-1)h_{ij,k},\\
(R_{ij,kl})'&=(R_{ij}')_{,kl}-\lambda(n-1)h_{ij,kl},\\
(\triangle R_{ij})'&=(\triangle R_{ij}')-\lambda(n-1)\triangle h_{ij},\\
(\triangle R)'&=\triangle(trRic')-\lambda(n-1)\triangle H.
\end{align*}\label{curvature derivative}
\end{prop}
From Proposition \ref{curvature derivative}, we calculate the second variations on TT direction.
\begin{align*}
\int_M\big(R_{i}^{\ plk}R_{jplk}\big)'h^{ij}dVg&=\int_M\Big((R_{i}^{\ plk})'R_{jplk}+R_{i}^{\ plk}(R_{jplk})'\Big)h^{ij}dV_g\\
&=\int_M\big(-h^{pm}R_{imkl}R_{jpkl}-h^{kn}R_{ipkl}R_{jpnl}-h^{ls}R_{ipks}R_{jpkl}\\     &\quad+2R_{jkpl}(R_{ikpl})'\Big)h^{ij}dV_g\\
&=\int_M\Big((4-2n)\lambda^2h_{ij}+2R_{jpkl}\big(h_{iq}R_{qpkl}\\
&\quad+\frac{1}{2}(h_{pi,lk}+h_{li,pk}-h_{pl,ik}-h_{pi,kl}
-h_{ki,pl}+h_{pk,il})\big)\Big)h^{ij}dV_g\\
&=\int_M\big(2(n+1)\lambda^2|h|^2-2\lambda h\triangle h\big)dV_g.
\end{align*}
By the same way, we can get the following formulae. Here we omit the detailed calculation.
\begin{align*}
\int_M(\triangle R_{ij})'h^{ij}dV_g&=\int_M\big(-\frac{1}{2}h\triangle^2h+\lambda h\triangle h)dV_g,\\
\int_M(R_{,ij})'h^{ij}dV_g&=0,\\
\int_M(R_{i}^{l}R_{jl})'h^{ij}dV_g&=\int_M(\lambda^2(n^2-1)|h|^2-\lambda(n-1) h\triangle h)dV_g,\\
\int_M(R^{pl}R_{ipjl})'h^{ij}dV_g&=\int_M\big((n^2-n-1)\lambda^2|h|^2-\frac{1}{2}\lambda(n-2)h\triangle h\big)dV_g,\\
\int_M(|Rm|^2g_{ij})'h^{ij}dV_g&=\int_M2\lambda^2n(n-1)|h|^2dV_g,\\
\int_M(\triangle Rg_{ij})'h^{ij}dV_g&=0,\\
\int_M(|Ric|^2g_{ij})'h^{ij}dV_g&=\int_M\lambda^2n(n-1)^2|h|^2dV_g,\\
\int_M(RR_{ij})'h^{ij}dV_g&=\int_M\big(\lambda^2n^2(n-1)|h|^2-\frac{1}{2}\lambda n(n-1)h\triangle h\big)dV_g,\\
\int_M(R^2g_{ij})'h^{ij}dV_g&=\int_M\lambda^2n^2(n-1)^2|h|^2dV_g.\\
\end{align*}

Following \eqref{secondvariation} and the above equations, we can get the second variations of  $\mathcal{F}_{s,\tau}$  in TT direction.
\begin{thm}Let $(M,g)$ be a compact Riemannian manifold with constant curvature satisfying \eqref{Riemcurv} and $h$ a TT tensor. Then the second variation of $\mathcal{F}_{s,\tau}$ is
\begin{align*}
\frac{d^2}{dt^2}\bigg|_{t=0}&\mathcal{F}_{s,\tau}
=\int_M\bigg\langle\Big(2(\triangle_L+2(n-1)\lambda)(\triangle_L+(n+2)\lambda)h
\bigg.\\  \bigg.&\quad+\frac{1}{2}(\triangle_L+2(n-1)\lambda)(s\triangle_L+(n-1)(4s+2n\tau)\lambda h\Big)\bigg\rangle dV_g
\\&=\int_M\left\langle\Big(\big(\triangle_L+2(n-1)\lambda\big)\big(\frac{4+s}{2}\triangle_L+\lambda\big(2n+4+(n-1)(2s+n\tau)\lambda\big)\Big )h,h\right\rangle dV_g.
\tag{3.7}\label{TTvariation}
\end{align*}
\end{thm}
When $\lambda=0$, \eqref{TTvariation} becomes
\begin{align*}
\frac{d^2}{dt^2}\bigg|_{t=0}\mathcal{F}_{s,\tau}
=\int_M2(1+\frac{s}{4})h\triangle^2hdV_g.\tag{3.8}\label{TT0}
\end{align*}
 By \eqref{TT0}, we can get the following corollary easily.
 
\begin{coro} Let $(M,g)$ be a compact manifold  with constant sectional curvature $\lambda=0$. Then the second variation of $\mathcal{F}_{s,\tau}$ at $g$ is non-negative as $s>-4$ and non-positive as $s<-4$ when the variation is restricted on TT directions.
\end{coro}

 Now, we focus on the case $\lambda\neq 0$ and finish the proof of Theorems 1.1 and 1.2. Firstly, if $\lambda>0$, we set $\lambda=1$. We then have
\begin{align*}
\frac{d^2}{dt^2}\bigg|_{t=0}\mathcal{F}_{s,\tau}
&=\int_M\bigg\langle\Big((\triangle_L+2(n-1))\big(\frac{4+s}{2}\triangle_L+(2n+4)+(n-1)(2s+n\tau)\big)\Big)h,h\bigg\rangle dV_g.\tag{3.9}\label{TT1}
\end{align*}
\begin{proof}[Proof of Theorem 1.1] By Proposition \ref{LLposieigen}, we know the least eigenvalue of the Lichnerowicz Laplacian on TT tensors is $4n$. Let
$-\triangle_Lh_{ij}=\lambda_Lh_{ij}$, and rewrite \eqref{TT1} as 
\begin{align*}
\frac{d^2}{dt^2}\bigg|_{t=0}\mathcal{F}_{s,\tau}
&=\int_M\Big(\big(\lambda_L-2(n-1)\big)\big(\frac{4+s}{2}\lambda_L-2n-4-(n-1)(2s+n\tau)\big)\Big)|h|^2 dV_g.\tag{3.10}\label{TT11}
\end{align*}
The first term $\lambda_L-2(n-1)>0$. We consider the second term. If $s>-4$, $\tau<\frac{6n-12}{n(n-1)}$,
\begin{align*}
\frac{4+s}{2}\lambda_L-(2n+4)-(n-1)(2s+n\tau)&>\frac{4+s}{2}4n-(2n+4)-(n-1)(2s+n\tau)\\
&>6n-12-n(n-1)\tau
\\&>0,
\end{align*}
\quad So, in this case we have $\frac{d^2}{dt^2}\bigg|_{t=0}\mathcal{F}_{s,\tau}\geq 0$.
If $s<-4$, $\tau>\frac{6n-12}{n(n-1)}$, 
\begin{align*}
\frac{4+s}{2}\lambda_L-(2n+4)-(n-1)(2s+n\tau)&<\frac{4+s}{2}4n-(2n+4)-(n-1)(2s+n\tau)\\
&<6n-12-n(n-1)\tau
\\&<0.
\end{align*}

 Therefore, in this case we have $\frac{d^2}{dt^2}\bigg|_{t=0}\mathcal{F}_{s,\tau}\leq 0$.
\end{proof}

\begin{proof}[Proof of Theorem 1.2]
 If $\lambda<0$, we set $\lambda=-1$. We have
\begin{align*}
\frac{d^2}{dt^2}\bigg|_{t=0}\mathcal{F}_{s,\tau}
&=\int_M\Big((\triangle_L-2(n-1))\big(\frac{4+s}{2}\triangle_L-((2n+4)-(n-1)(2s+n\tau)\big)\Big)hdV_g
\\&=\int_M\Big((\lambda_L+2(n-1))\big(\frac{4+s}{2}\lambda_L+((2n+4)+(n-1)(2s+n\tau)\big)\Big)|h|^2dV_g.
\tag{3.11}\label{TT-1}
\end{align*}
By Proposition \ref{negeigen}, the first term $\lambda_L+2(n-1)>0$. We  now consider the second term. If $s>-4$, and $\tau>\frac{6n-12}{n(n-1)}$,
\begin{align*}
\frac{4+s}{2}\lambda_L+(2n+4)+(n-1)(2s+n\tau)&>-\frac{4+s}{2}n+2n+4+(n-1)(2s+n\tau)\\
&>-6n+12+n(n-1)\tau
\\&>0.
\end{align*}
In this case we have $\frac{d^2}{dt^2}\big|_{t=0}\mathcal{F}_{s,\tau}\geq 0$. By the same way we know $\frac{d^2}{dt^2}\big|_{t=0}\mathcal{F}_{s,\tau}\geq 0$ when $s<-4$, $\tau<\frac{6n-12}{n(n-1)}$.
 \end{proof}

\subsection{Conformal variations}
Now, we consider the conformal variations of the functionals $\mathcal{F}_{s,\tau}$ at a Riemannian metric with constant sectional curvature satisfying \eqref{Riemcurv}. Let $\mathcal{M}_1[g]$ denote the space of unit volume metrics conformal to $g$. The tangent space of $\mathcal{M}_1[g]$ consists of functionals with mean value zero.
\begin{prop}If $g$ has constant sectional curvature, satisfying \eqref{Riemcurv} and \eqref{Riccurv}, and let  $h=fg$,
\begin{align*}
(R_{ij})'&=-\frac{1}{2}(n-2)f_{,ij}-\frac{1}{2}\triangle fg_{ij},\\
(R_{ij,kl})'&=(R_{ij}')_{,kl}-\lambda(n-1)f_{,kl}g_{ij},\\
(\triangle R_{ij})'&=(\triangle R_{ij}')-\lambda(n-1)\triangle fg_{ij},\\
(\triangle R)'&=-(n-1)\triangle^2f-\lambda n(n-1)\triangle f.
\end{align*}\label{curvaturederivative}
\end{prop}
Following proposition \eqref{curvaturederivative}, we have,
\begin{align*}
\int_M\big(R_{i}^{\ plk}R_{jplk}\big)'fg^{ij}dVg&=\int_M\Big((R_{i}^{\ plk})'R_{jplk}+R_{i}^{\ plk}(R_{jplk})'\Big)h^{ij}dV_g,\\
&=\int_M\big(-fg^{pm}R_{imkl}R_{jpkl}-fg^{kn}R_{ipkl}R_{jpnl}-fg^{ls}R_{ipks}R_{jpkl}\\     &\quad+2R_{jkpl}(R_{ikpl})'\Big)fg^{ij}dV_g,\\
&=\int_M\big(-3|Rm|^2f^2+2fR^{ikjl}R^{'}_{ikjl}\big)dV_g,\\
&=\int_M\big(-3|Rm|^2f^2+2\lambda f(g^{ij}g^{kl}-g^{il}g^{kj})R^{'}_{ikjl}\big)dV_g,\\
&=\int_M\Big(-3|Rm|^2f^2+4\lambda fg^{ij}\big((g^{kl}R_{ikjl})^{'}+fg^{kl}R_{ikjl}\big)\Big)dV_g,\\
&=\int_M\Big(-2\lambda^2n(n-1)f^2-4\lambda(n-1)f\triangle f\Big)dV_g.\\
\end{align*}
By the same way, we give the following formulas.
\begin{align*}
\int_M(\triangle R_{ij})'fg^{ij}dV_g&=\int_M\Big(-(n-1)f\triangle^2f-\lambda n(n-1)f\triangle f\Big)dV_g,\\
\int_M(R_{,ij})'fg^{ij}dV_g&=\int_M\Big(-(n-1)f\triangle^2f-\lambda n(n-1)f\triangle f\Big)dV_g,\\
\int_M(R_{i}^{l}R_{jl})'fg^{ij}dV_g&=\int_M\Big(-\lambda^2n(n-1)^2f^2-2\lambda(n-1)^2f\triangle f\Big)dV_g,\\
\int_M(R^{pl}R_{ipjl})'fg^{ij}dV_g&=\int_M\Big(-\lambda^2n(n-1)^2f^2-2\lambda(n-1)^2f\triangle f\Big)dV_g,\\
\int_M(|Rm|^2g_{ij})'fg^{ij}dV_g&=\int_M\Big(-2\lambda^2n^2(n-1)f^2-4\lambda n(n-1)f\triangle f\Big)dV_g,\\
\int_M(\triangle Rg_{ij})'fg^{ij}dV_g&=\int_M\Big(-n(n-1)f\triangle^2f-\lambda n^2(n-1)f\triangle f\Big)dV_g,\\
\int_M(|Ric|^2g_{ij})'fg^{ij}dV_g&=\int_M\Big(-\lambda^2n^2(n-1)^2f^2-2\lambda n(n-1)^2f\triangle f\Big)dV_g,\\
\int_M(RR_{ij})'fg^{ij}dV_g&=\int_M\Big(-\lambda^2n^2(n-1)^2f^2-2\lambda n(n-1)^2f\triangle f\Big)dV_g,\\
\int_M(R^2g_{ij})'fg^{ij}dV_g&=\int_M\Big(-\lambda^2n^3(n-1)^2f^2-2\lambda n^2(n-1)^2f\triangle f\Big)dV_g,\\
\end{align*}

Using the above equalities, we get the conformal variations for $\mathcal{F}_{s,\tau}$ at metrics with constant sectional curvature.
\begin{thm}Let $(M,g)$ is a compact Riemannian manifold with constant curvature satisfying \eqref{Riemcurv} and $h=fg$ with $\int_MfdV_g=0$. Then the second variation of $\mathcal{F}_{s,\tau}$ is 
\begin{align*}
\frac{d^2}{dt^2}\bigg|_{t=0}\mathcal{F}_{s,\tau}
&=\int_M\Big(2(n-1)f\triangle^2f+8\lambda(n-1)f\triangle f-2\lambda^2(n^2-n)(n-4)f^2\\
&+s\big(\frac{1}{2}n(n-1)f\triangle^2f-\frac{1}{2}\lambda(n-1)(n^2-10n+8)f\triangle f-\lambda^2n(n-1)^2(n-4)f^2\big)\\
&+\tau\big(2(n-1)^2f\triangle^2f-\lambda n(n-1)^2(n-6)f\triangle f-\lambda^2n^2(n-1)^2(n-4)f^2\big)\Big)dVg.\tag{3.12}\label{CFvariation}
\end{align*}
\end{thm}

When $\lambda=0$, \eqref{CFvariation} becomes
\[
\frac{d^2}{dt^2}\bigg|_{t=0}\mathcal{F}_{s,\tau}=\int_M\frac{1}{2}(n-1)(sn+4(n-1)\tau+4)f\triangle^2fdVg.\tag{3.13}\label{CF0}
\]
We then have 
\begin{coro} Let $(M,g)$ be a compact manifold with constant sectional curvature $\lambda=0$. Then the second variation of $\mathcal{F}_{s,\tau}$ on $(M,g)$ is nonnegative as $s+4\tau>\frac{4}{n}(\tau-1)$ when the variation is  restricted on the conformal direction.
\end{coro}

Next we consider the case $\lambda\neq 0$ and finish the proof of Theorems 1.5 and 1.6. Firstly, if $\lambda>0$, we set $\lambda=1$. We then have 
\begin{align*}
\frac{d^2}{dt^2}\bigg|_{t=0}\mathcal{F}_{s,\tau}
=\int_M\bigg\langle\Big(n&-1\Big)\Big(\triangle+n\Big)\Big(\frac{ns-4\tau+4n\tau+4}{2} \triangle \bigg.\\ \bigg. &-\big(n-4\big)\big(n^2\tau+ns-n\tau-s+2\big)\Big)f ,f\bigg\rangle dV_g.\tag{3.14}\label{CF1}
\end{align*}

\begin{prop}(\cite{Aubin1998}) If $(M^n, g)$ is a compact manifold  satisfying
\[
Ric\geq(n-1)\cdot g,
\]
then the lowest non-trivial eigenvalue satisfies $\lambda_1\geq n$, and the equality holds if and any if $(M^n, g)$ is isometric to $(S^n, g_s)$. \label{CFeigen}
\end{prop}

 Let $\mu$ be a non-zero eigenvalue of $\triangle$. We consider the following polynomial
\[
P_1(\mu)=\big(n-1\big)\big(\mu-n\big)\big(\frac{ns-4\tau+4n\tau+4}{2}\mu+(n-4)(n^2\tau+ns-n\tau-s+2)\big).\tag{3.15}\label{CFP}
\]

Since the variation is restricted to $\mathcal{M}_1[g]$,  we should to prove that the second variation of the functional on functions with mean value zero is non-negative (or non-positive in the maximizing case).

\begin{proof}
[Proof of Theorem 1.5]
 By Proposition \eqref{CFeigen}, the second term $\mu-n\geq 0$. Then the sign of the second variations for $\mathcal{F}_{s,\tau}$ is determined by the third term. If $n=4$,
\[
P_1(\mu)=6\big(\mu-4\big)(s+3\tau+1)\mu.\tag{3.16}\label{CFP4}
\]
Form \eqref{CFP4}, the first part of Theorem 1.5 can get easily.

 If $n=3$,
\[
P_1(\mu)=2(\mu-3)\big(\frac{3s+8\tau+4}{2}\mu-(6\tau+2s+2)\big).\tag{3.17}\label{{CFP3}}
\]
 The functional will be minimizing if $\frac{3s+8\tau+4}{2}>0$ and
\[
\frac{3s+8\tau+4}{2}\mu-(6\tau+2s+2)\geq\frac{5}{2}s+6\tau+4>0.
\]
 Then $P_1(\mu)\geq 0$ if the following hold:
\begin{equation*}
\left\{
\begin{array}{lr}
\tau<1 &\\
s>-\frac{8}{3}\tau-\frac{4}{3}&
\end{array}
\right.
\quad or\quad
\left\{
\begin{array}{lr}
\tau>1&\\
s>-\frac{12}{5}\tau-\frac{8}{5}.&
\end{array}
\right.
\end{equation*}

 Similarly we can get $P_1(\mu)\leq 0$ if the following hold:
 \begin{equation*}
\left\{
\begin{array}{lr}
\tau<1 &\\
s<-\frac{12}{5}\tau-\frac{8}{5}&
\end{array}
\right.
\quad or\quad
\left\{
\begin{array}{lr}
\tau>1&\\
s<-\frac{8}{3}\tau-\frac{4}{3}.&
\end{array}
\right.
\end{equation*}

When $n\geq 5$, the functional will be minimizing if $ns-4\tau+4n\tau+4>0$ and
\[
\dfrac{ns-4\tau+4n\tau+4}{2}n+(n-4)(n^2\tau+ns-n\tau-s+2)>0.
\]
Then $P_1(\mu)\geq 0$ if the following hold
\begin{equation*}
\left\{
\begin{array}{lr}
\tau>\frac{2}{(n-1)(n-2)} &\\
s>-\frac{4(n-1)}{n}\tau-\frac{4}{n}&
\end{array}
\right.
\quad or\quad
\left\{
\begin{array}{lr}
\tau<\frac{2}{(n-1)(n-2)}&\\
s>-\frac{2n(n-1)}{3n-4}\tau-\frac{8}{3n-4}.&
\end{array}
\right.
\end{equation*}
And we can get $P_1(\mu)\leq 0$ if the following hold
\begin{equation*}
\left\{
\begin{array}{lr}
\tau>\frac{2}{(n-1)(n-2)}&\\
s<-\frac{2n(n-1)}{3n-4}\tau-\frac{8}{3n-4}
\end{array}
\right.
\quad or\quad
\left\{
\begin{array}{lr}
\tau<\frac{2}{(n-1)(n-2)}&\\
s<-\frac{4(n-1)}{n}\tau-\frac{4}{n}.&
\end{array}
\right.
\end{equation*}
\end{proof}

Finally we consider the case $\lambda<0$. We set $\lambda=-1$. 
\begin{align*}
\frac{d^2}{dt^2}\bigg|_{t=0}\mathcal{F}_{s,\tau}
=\int_M\bigg\langle\big(n-1\big)\big(\triangle&-n\big)\big(\frac{ns-4\tau+4n\tau+4}{2}\triangle\bigg.\\ \bigg.&+(n-4)(n^2\tau-n\tau-s+2)\big)f ,f\bigg\rangle dV_g.\tag{3.18}\label{CF-1}
\end{align*}
Let $\mu$ is a non-zero eigenvalue of $\triangle$, and consider the following polynomial
\[
P_2(\mu)=\big(n-1\big)\big(\mu+n\big)\big(\frac{ns-4\tau+4n\tau+4}{2}\mu-(n-4)(n^2\tau+ns-n\tau-s+2)\big).\tag{3.19}\label{CF-1P}
\]
Since the first term of $P_2(\mu)$ is always nonnegative, then we just consider the second term.
Let $n=3$, then $P_2(\mu)$ is nonnegative if
\begin{equation*}
\left\{
\begin{array}{lr}
\tau>2 &\\
s>-\frac{8}{3}\tau-\frac{4}{3}&
\end{array}
\right.
\quad or\quad
\left\{
\begin{array}{lr}
\tau<2&\\
s>-3\tau-2.&
\end{array}
\right.
\end{equation*}
And $P_2(\mu)$ is non-positive if
\begin{equation*}
\left\{
\begin{array}{lr}
\tau>2 &\\
s<-3\tau-2&
\end{array}
\right.
\quad or\quad
\left\{
\begin{array}{lr}
\tau<2&\\
s<-\frac{8}{3}\tau-\frac{4}{3}.&
\end{array}
\right.
\end{equation*}

When $n\geq 4$, do it in the same way as Theorem 1.5. Then we finish the proof of Theorem 1.6.

\textbf{Addresses:}

\vspace{2mm} Weimin Sheng: School of Mathematical Sciences, Zhejiang University,
Hangzhou 310027, China.

\vspace{2mm} Lisheng Wang: School of Mathematical Sciences, Zhejiang university,
Hangzhou 310027, China.

\vspace{3mm}

\textbf{Email:} weimins@zju.edu.cn; wlspaper@yeah.net

\end{document}